\documentclass[12pt]{amsart}

\usepackage{amssymb,amsmath,amsthm,enumerate}
\numberwithin{equation}{section}
\newtheorem{thm}{Theorem}[section]
\newtheorem{lem}[thm]{Lemma}

\newtheorem*{definition}{Definition}
\newtheorem{ex}[thm]{Example}
\newtheorem{rem}[thm]{Remark}

\newcommand{\hp}{\mathbb{H}}
\newcommand{\Z}{\mathbb{Z}}

\newcommand{\C}{\mathbb{C}}
\newcommand{\Q}{\mathbb{Q}}
\newcommand{\F}{\mathbb{F}}

\newcommand{\SL}{\operatorname{SL}_2(\Z)}

\renewcommand{\mod}[1]{\,(\mathrm{mod}\,{#1})}

\newcommand{\leg}[2]{\left(\frac{#1}{#2}\right)}

\title[Non-existence of simple congruences for quotients of Eisenstein series]{On the non-existence of simple congruences for quotients of Eisenstein series}
\author{Michael Dewar}

\begin{document}

\begin{abstract}
A recent article of Berndt and Yee found congruences modulo $3^{k}$ for certain ratios of Eisenstein series.  For all but one of these, we show there are no simple congruences $a(\ell n+c)\equiv 0\mod \ell$ when $\ell\geq 13$ is prime.  This follows from a more general theorem on the non-existence of congruences in $E_{2}^{r}E_{4}^{s}E_{6}^{t}$ where $r\geq 0$ and $s,t\in\Z$.
\end{abstract}

\maketitle

\section{Introduction}
Define $p(n)$ to be the number of ways of writing $n$ as a sum of non-increasing positive integers.  Ramanujan famously established the congruences
\begin{align*}
p(5n+4)&\equiv 0\mod 5\\
p(7n+5)&\equiv 0\mod 7\\
p(11n+6)&\equiv 0\mod {11}
\end{align*}
and noted that there does not appear to be any other prime for which the partition function has equally simple congruences.  Ahlgren and Boylan~\cite{ahlboy2003a} build on the work of Kiming and Olsson~\cite{kimols1992a} to prove that there truly are no other such primes.  For large enough primes $\ell$, Sinick~\cite{sin2009a} and the author~\cite{dew2009a} prove the non-existence of simple congruences
\begin{align*}
a(\ell n+c) \equiv 0\mod \ell
\end{align*}
for wide classes of functions $a(n)$ related to the coefficients of modular forms.  However, all of the modular forms studied in \cite{ahlboy2003a}, \cite{sin2009a} and \cite{dew2009a} are non-vanishing on the upper half plane.  Here we prove the non-existence of simple congruences (when $\ell$ is large enough) for ratios of Eisenstein series.

Let $\sigma_{m}(n) := \sum_{d|n} d^{m}$ and define the Bernoulli numbers $B_{k}$ by $\frac{t}{e^{t}-1} = \sum_{k=0}^{\infty} B_{k}\frac{t^{k}}{k!}$.  For even $k\geq 2$, set
\begin{align*}
E_{k}(\tau) &:= 1- \frac{2k}{B_{k}} \sum_{n=1}^{\infty} \sigma_{k-1}(n) q^{n}.
\end{align*}
Note that $E_{2}\equiv E_{4} \equiv E_{6}\equiv 1$ modulo 2 and 3.
Berndt and Yee \cite{beryee2002a} prove congruences for the quotients of Eisenstein series in Table~\ref{BYsummary} below, where $F(q):=\sum a(n)q^{n}$.
\begin{table}[t]\label{BYsummary}
\caption{Congruences of Berndt and Yee~\cite{beryee2002a}}
\begin{center}
\begin{tabular}{|c|c|c|}\hline
$F(q)$ & $n\equiv 2 \mod 3$ & $n\equiv 4\mod 8$\\\hline\hline
$1/E_{2}$ & $a(n)\equiv 0 \mod{3^{4}}$ & \\\hline
$1/E_{4}$ & $a(n)\equiv 0 \mod{3^{2}}$ & \\\hline
$1/E_{6}$ & $a(n)\equiv 0 \mod{3^{3}}$ & $a(n)\equiv 0 \mod{7^{2}}$\\\hline
$E_{2}/E_{4}$ & $a(n)\equiv 0 \mod{3^{3}}$ & \\\hline
$E_{2}/E_{6}$ & $a(n)\equiv 0 \mod{3^{2}}$ & $a(n)\equiv 0 \mod{7^{2}}$\\\hline
$E_{4}/E_{6}$ & $a(n)\equiv 0 \mod{3^{3}}$ & \\\hline
$E_{2}^{2}/E_{6}$ & $a(n)\equiv 0 \mod{3^{5}}$ & \\\hline
\end{tabular}
\end{center}
\label{default}
\end{table}
An obviously necessary requirement for the congruences in the $n\equiv2\mod 3$ column of Table~\ref{BYsummary} is that there are simple congruences of the form $a(3n+2)\equiv 0\mod 3$.  All but the first form in Table~\ref{BYsummary} are covered by the following theorem.

\begin{thm}\label{main thm}
Let $r\geq 0$ and $s,t\in\Z$.  If $E_{2}^{r}E_{4}^{s}E_{6}^{t}=\sum a(n)q^{n}$ has a simple congruence $a(\ell n+c)\equiv 0 \mod \ell$ for the prime $\ell$, then either $\ell \leq 2r +8|s|+12|t| + 21$ or $r=s=t=0$.
\end{thm}
This theorem gives an explicit upper bound on primes $\ell$ for which there can be congruences of the form $a(\ell n +c) \equiv 0 \mod{\ell^{k}}$ as in the middle column of Table~\ref{BYsummary}.

\begin{rem}
See Remark~\ref{bound improvement} for a slight improvement of Theorem~\ref{main thm} in some cases.
\end{rem}

\begin{ex}
The form ${E_{6}}/{E_{4}^{12}}$ can only have simple congruences for $\ell\leq 129$.  Of these, the primes $\ell=2$ and $3$ are trivial with $E_{4}\equiv E_{6} \equiv 1 \mod \ell$.  For the remaining primes, the only congruences are 
\begin{align*}
a(\ell n + c)\equiv 0\mod {17}, \text{ where } \leg{c}{17}=-1.
\end{align*}
\end{ex}

Mahlburg~\cite{mah2005a} shows that for each of the forms in Table~\ref{BYsummary} except $1/E_{2}$, there are infinitely many primes $\ell$ such that for any $i\geq 1$, the set of $n$ with $a(n)\equiv 0\mod {\ell^{i}}$ has arithmetic density 1.  On the other hand, our result shows that (for large enough $\ell$) every arithmetic progression modulo $\ell$ has at least one non-vanishing coefficient modulo $\ell$.

Section 2 recalls certain definitions and tools from the theory of modular forms.  Simple congruences are reinterpreted in terms of Tate cycles, which are reviewed in Section 3.  Section 4 proves Theorem~\ref{main thm}.

{\it Acknowledgments}:   The author would like to thank Scott Ahlgren for careful readings of this article and many helpful suggestions.

\section{Preliminaries}\label{section: Preliminaries}

A modular form of weight $k\in\Z$ on $\SL$ is a holomorphic function $f: \hp \to \C$ which satisfies
\begin{align*}
f\left(\frac{a\tau +b}{c\tau + d}\right) = (c\tau + d)^{k} f(\tau)
\end{align*}
for every $\left(\begin{array}{cc}a & b \\c & d\end{array}\right)\in\SL$, and which is holomorphic at infinity.  Modular forms have Fourier expansions in powers of  $q=e^{2\pi i \tau}$. For any prime $\ell\geq 5$, let $\Z_{(\ell)}=\{\frac{a}{b} \in \Q: \ell \nmid b\}$.  We denote the set of all weight $k$ modular forms on $\SL$ with $\ell$-integral Fourier coefficients by $M_{k}$.  Although $E_{k}$ is a modular form of weight $k$ whenever $k\geq 4$, $E_{2}$ is called a quasi-modular form since it satisfies the slightly different transformation rule
\begin{align*}
E_{2}\left(\frac{a\tau +b}{c\tau + d}\right) = (c\tau + d)^{2} E_{2}(\tau) -\frac{6ic}{\pi}(c\tau + d).
\end{align*}

\begin{definition}
If $\ell$ is a prime, a Laurent series $f = \sum_{n\geq N} a(n)q^{n}\in \Z_{(\ell)}(\!(q)\!)$ has a {\it simple congruence} at $c\mod\ell$ if $a(\ell n +c) \equiv 0\mod\ell$ for all $n$.
\end{definition}

\begin{lem}\label{replacement lemma}
Suppose that $\ell$ is prime and that $f= \sum a(n)q^{n}$ and $g=\sum b(n)q^{n}\in\Z_{(\ell)}(\!(q)\!)$ with $g\not\equiv 0\mod\ell$.  The series $f$ has a simple congruence at $c\mod\ell$ if and only if the series $fg^{\ell}$ has a simple congruence at $c\mod\ell$.
\end{lem}
\begin{proof}
It suffices to consider the reductions $\mod\ell$ of the series
\begin{align*}
\left(\sum a(n)q^{n}\right)\left(\sum b(n)q^{\ell n}\right) \equiv \sum_{n} \left( \sum_{m} b(m) a(n-\ell m) \right)q^{n} \mod\ell.
\end{align*}
If $a(n)$ vanishes when $n\equiv c\mod\ell$, then the inner sum on the right hand side will also vanish for $n\equiv c\mod\ell$.  The converse follows via multiplication by $\left(\sum b(n)q^{n}\right)^{-\ell}$ and repetition of this argument.
\end{proof}

Our main tool is Ramanujan's $\Theta$ operator
\begin{align*}
\Theta := \frac{1}{2\pi i} \frac{d}{d\tau} =  q \frac{d}{dq}.
\end{align*}
For any prime $\ell$ and any Laurent series $f= \sum a(n)q^{n}\in\Z_{(\ell)}(\!(q)\!)$, by Fermat's Little Theorem
\begin{align*}
\Theta^{\ell} f = \sum a(n) n^{\ell} q^{n} \equiv \sum a(n) n q^{n} = \Theta f \mod \ell.
\end{align*}
We call the sequence $\Theta f, \dots, \Theta^{\ell} f \mod \ell$ the Tate cycle of $f$. Note that $\Theta^{\ell-1} f \equiv f \mod \ell$ is equivalent to $f$ having a simple congruence at $0\mod\ell$.

We now recall some facts about the reductions of modular forms $\mod\ell$.  See Swinnerton-Dyer~\cite{swi1973a} Section 3 for the details on this paragraph.  There are polynomials $A(Q,R), B(Q,R) \in \Z_{(\ell)}[Q,R]$ such that
\begin{align*}
A(E_{4}, E_{6}) &=E_{\ell-1},\\
B(E_{4},E_{6}) &=E_{\ell+1}.
\end{align*}
Reduce the coefficients of these polynomials modulo $\ell$ to get $\tilde{A}, \tilde{B}\in \F_{\ell}[Q,R]$.  Then $\tilde{A}$ has no repeated factor and is prime to $\tilde{B}$.  Furthermore, the $\F_{\ell}$-algebra of reduced modular forms is naturally isomorphic to
\begin{align}\label{iso}
\frac{\F_{\ell}[Q, R]}{\tilde{A}-1}
\end{align}
via $Q\to E_{4}$ and $R\to E_{6}$.  Whenever a power series $f$ is congruent to a modular form, define the filtration of $f$ by
\begin{align*}
\omega(f) := \inf\{k : f\equiv g \in M_{k} \mod\ell\}.
\end{align*}
If $f\in M_{k}$, then for some $g\in M_{k+\ell +1}$, $\Theta f \equiv g \mod\ell$.  The next lemma also follows from~\cite{swi1973a} Section 3.

\begin{lem}\label{filtration of theta f}
Let $\ell \geq 5$ be prime, $f\in M_{k_{1}}$, $f\not\equiv 0 \mod \ell$ and $g\in M_{k_{2}}$.
\begin{enumerate}
\item If $f\equiv g\mod \ell$ then $k_1 \equiv k_2 \mod {\ell-1}$,
\item $\omega(\Theta f) \leq \omega(f) + \ell + 1$ with equality if and only if $\omega(f)\not\equiv 0\mod \ell$,
\item If $\omega(f)\equiv 0\mod\ell$, then for some $s\geq 1$,  $\omega(\Theta f) = \omega(f) +(\ell + 1) -s(\ell-1)$, and
\item $\omega(f^{i})=i\omega(f)$.
\end{enumerate}
\end{lem}
The natural grading induced by (\ref{iso}) provides a key step in the following lemma which is taken from the proof of \cite{kimols1992a} Proposition 2.
\begin{lem}\label{unex cong lem}
A form $f\in M_k$ with $\Theta f \not\equiv 0\mod\ell$ has a simple congruence at $c\not\equiv 0\mod \ell$ if and only if  $\Theta^\frac{\ell +1}{2}f\equiv -\leg{c}{\ell}\Theta f \mod \ell$.
\end{lem}

\begin{proof}
Since $\Theta$ satisfies the product rule,
\begin{align*}
\Theta^{\ell-1}\left(q^{-c}f\right) &\equiv \sum_{i=0}^{\ell -1} \binom{\ell-1}{i}(-c)^{\ell-1-i}q^{-c}\Theta^i f \mod \ell\\
&\equiv \sum_{i=0}^{\ell -1} c^{\ell-1-i}q^{-c}\Theta^i f \mod \ell\\
&\equiv c^{\ell-1}q^{-c}f + \sum_{i=1}^{\ell -1} c^{\ell-1-i}q^{-c}\Theta^i f \mod \ell.
\end{align*}
A simple congruence for $f$ at $c\not\equiv 0\mod \ell$ is equivalent to a simple congruence for $q^{-c}f$ at $0\mod\ell$, which in turn is equivalent to $\Theta^{\ell-1}\left(q^{-c}f\right) \equiv q^{-c}f \mod\ell$.  By the computation above, this is equivalent to $0\equiv \sum_{i=1}^{\ell -1} c^{\ell-1-i}q^{-c}\Theta^i f \mod \ell$, and hence to $0\equiv \sum_{i=1}^{\ell -1} c^{\ell-1-i}\Theta^i f \mod \ell$.  By Lemma \ref{filtration of theta f} (2) and (3), for $1\leq i \leq \frac{\ell-1}{2}$ we have
\begin{equation*}
\omega(\Theta^i f) \equiv \omega(\Theta^{i +\frac{\ell-1}{2}} f) \equiv \omega(f) +2i \mod{\ell-1}.
\end{equation*}
By Lemma \ref{filtration of theta f} (1) and the natural grading (filtration modulo $\ell-1$), the only way for the given sum to be zero is if for all $1\leq i\leq \frac{\ell-1}{2}$ we have
\begin{equation*}
c^{\ell -1-i}\Theta^i f +c^{\ell -1-(i+\frac{\ell-1}{2})}\Theta^{i+ \frac{\ell-1}{2}}f \equiv 0 \mod\ell,
\end{equation*}
which happens if and only if
\begin{align*}
\Theta^{i + \frac{\ell-1}{2}}f &\equiv -c^\frac{\ell-1}{2}\Theta^i f \equiv - \leg{c}{\ell}\Theta^i f \mod \ell,
\end{align*}
which happens if and only if
\begin{align*}
\Theta^{\frac{\ell+1}{2}}f \equiv - \leg{c}{\ell} \Theta f \mod\ell.
\end{align*}
\end{proof}

\begin{lem}\label{eis filt}
Let $a,b,c\geq 0$ be integers and let $\ell>11$ be prime.  Then $\omega(E_{\ell+1}^{a}E_{4}^{b}E_{6}^{c}) = a\ell + a +4b+6c$.
\end{lem}
\begin{proof}
Since $E_{\ell+1}^{a}E_{4}^{b}E_{6}^{c}\in M_{a\ell + a +4b+6c}$, it suffices to show that $\tilde{A}(Q,R)$ does not divide $\tilde{B}(Q,R)^{a}Q^{b}R^{c}$.  However $\tilde{A}$ has no repeated factors and is prime to $\tilde{B}$ and so it suffices to show that $\tilde{A}$ does not divide $QR$.  But $QR$ has weight 10 and $E_{\ell-1}$ has weight $\ell-1 > 10$ so this is impossible.
\end{proof}

\section{The Structure of Tate Cycles}\label{section: The Structure of Tate Cycles}
The following framework follows Jochnowitz \cite{joc1982a}.  Let $f\in M_k$ be such that $\Theta f\not\equiv 0\mod\ell$.  Recall from Section~\ref{section: Preliminaries} that the Tate cycle of $f$ is the sequence $\Theta f, \dots, \Theta^{\ell-1} f \mod\ell$.  By Lemma~\ref{filtration of theta f} (2) and (3),
\begin{align*}
\omega (\Theta^{i+1} f) \equiv \begin{cases} \omega(\Theta^{i}f) + 1 \mod \ell & \text{if } \omega(\Theta^{i} f) \not\equiv 0\mod\ell\\
s + 1 \mod \ell & \text{if } \omega(\Theta^{i} f)\equiv 0\mod\ell,
\end{cases}
\end{align*}
for some $s\geq 1$.  In particular, when $\omega(\Theta^{i}f) \equiv 0\mod \ell$, the amount $s$ by which the filtration decreases controls when the {\it next} decrease occurs.  We say that $\Theta^{i}f$ is a high point of the Tate cycle and $\Theta^{i+1}f$ is a low point of the Tate cycle whenever $\omega(\Theta^{i}f)\equiv 0\mod\ell$.   Elementary considerations (see, for example, \cite{joc1982a} Section 7 or \cite{dew2009a} Section 3) yield
\begin{lem}\label{tc basics lem}
Let $f\in M_{k}$ with $\Theta f \not\equiv 0\mod\ell$.
\begin{enumerate}
\item If the Tate cycle has only one low point, then the low point has filtration $2\mod\ell$. 
\item The Tate cycle has one or two low points.
\end{enumerate}
\end{lem}

\begin{lem}\label{what is s}
Suppose $f\in M_{k}$ has a simple congruence at $c\not\equiv 0\mod\ell$, where $\ell\geq 5$ is prime, and $\Theta f\not\equiv 0\mod\ell$.  Then the Tate cycle of $f$ has two low points.  Furthermore, if $\Theta^{i}f$ is a high point, then
\begin{align*}
\omega(\Theta^{i+1} f) = \omega(\Theta^{i}f) + (\ell + 1) - \left( \frac{\ell+1}{2} \right)(\ell-1) \equiv \frac{\ell+3}{2} \mod\ell.
\end{align*}
\end{lem}
\begin{proof}
By Lemma~\ref{unex cong lem}, $\omega\left(\Theta f\right) =\omega(\Theta^{\frac{\ell+1}{2}}f)$.  Hence, the filtration is not monotonically increasing between $\Theta f$ and $\Theta^{\frac{\ell+1}{2}}f$, so there must be a fall in filtration somewhere in the first half of the Tate cycle.  We also have $\omega(\Theta^{\frac{\ell+1}{2}}f)=\omega\left(\Theta f\right) =\omega\left(\Theta^{\ell} f\right)$ and so there must be a low point somewhere in the second half of the Tate cycle.  By Lemma~\ref{tc basics lem}, there are exactly two low points in the Tate cycle.  Lemma~\ref{filtration of theta f} (2) and (3) give
\begin{align*}
 \omega\left(\Theta f\right)=\omega\left( \Theta^{\frac{\ell+1}{2}}f \right) = \omega\left(\Theta f\right) + \left(\frac{\ell-1}{2}\right)(\ell+1) -s(\ell-1)
\end{align*}
for some $s\geq 1$.  Hence $s=\frac{\ell+1}{2}$.  The lemma follows.
\end{proof}
The proof of Theorem~\ref{main thm} uses the previous lemma to determine how far the filtration falls, and the bounds of the next lemma to show a corresponding restriction on $\ell$.
\begin{lem}\label{key bounds lemma}
Let $\ell\geq 5$ be prime and suppose $f\in M_{k}$ has a simple congruence at $c\not\equiv 0\mod\ell$.  If $\omega(f) = A\ell +B$ where $1\leq B \leq \ell-1$, then
\begin{align*}
\frac{\ell+1}{2} \leq B \leq A +\frac{\ell+3}{2}.
\end{align*}
\end{lem}
\begin{proof}
Since $B\neq 0$, $\omega(\Theta f) = (A+1)\ell + (B+1)$.  From the proof of Lemma~\ref{what is s}, the Tate cycle has a high point before $\Theta^{\frac{\ell+1}{2}}f$.  Hence by Lemma~\ref{filtration of theta f} (2),
\begin{align*}
B+1 + \frac{\ell-3}{2} \geq \ell,
\end{align*}
which gives the first inequality.  Also by Lemma~\ref{filtration of theta f}, the high point has filtration
\begin{align*}
\omega(\Theta^{\ell-B} f) &= \omega(f) + (\ell-B)(\ell+1)\\
&= (A+\ell-B+1)\ell.
\end{align*}
Lemma~\ref{what is s} implies that the corresponding low point has filtration
\begin{align*}
\omega(\Theta^{\ell-B+1} f) &= \left( A-B+\frac{\ell+3}{2} \right)\ell + \left(\frac{\ell+3}{2}\right).
\end{align*}
The fact that $\omega(\Theta^{\ell-B+1} f) \geq 0$ implies the second inequality.
\end{proof}

If $\Theta f \equiv 0 \mod \ell$ then the Tate cycle is trivial and above lemmas are not applicable.  We dispense with this case now.

\begin{lem}\label{theta is zero}
Let $f=E_{2}^{r}E_{4}^{s}E_{6}^{t}$ where $r\geq 0$ and $s,t\in \Z$.  If $\ell$ is a prime such that $\Theta f \equiv 0 \mod\ell$ then either $\ell \leq13$ or $r\equiv s\equiv t\equiv 0\mod \ell$.
\end{lem}

\begin{ex}\label{ell is 11}
We have $\Theta(E_{4}E_{6}) \equiv 0 \mod\ell$ for $\ell=2,3,11$.  
\end{ex}
\begin{ex}
We have $\Theta(E_{2}^{144}E_{4}^{-15}E_{6}^{-14})\equiv 0 \mod\ell$ for $\ell=2,3,5,7,13$.
\end{ex}

Note that $\Theta f\equiv 0 \mod\ell$ is equivalent to $f$ having simple congruences at all $c\not\equiv 0\mod\ell$.

\begin{proof}[Proof of Lemma~\ref{theta is zero}]
Assume $\ell \geq 17$ and expand $f$ as a power series to get
\begin{align*}
f = 1 + \Large(-24r &+240s -504t\Large)q + \Large(288r^2 - 5760rs + 12096rt \\&- 360r + 28800s^2 - 120960st - 26640s + 127008t^2 - 143640t\Large)q^2 +\cdots.
\end{align*}
If $\Theta f\equiv 0 \mod\ell$, then the coefficients of $q$ and $q^{2}$ vanish modulo $\ell$.  That is,
\begin{equation}\label{coeff q}
-24r +240s -504t\equiv 0 \mod\ell,
\end{equation}
and
\begin{equation}\label{coeff qq}
\begin{split}288r^2 - 5760rs + 12096rt- 360r + 28800s^2 \\- 120960st - 26640s + 127008t^2 - 143640t \end{split}\equiv 0 \mod\ell.
\end{equation}
Furthermore, by Lemmas~\ref{filtration of theta f}(2) and~\ref{eis filt} and the fact that $E_{2}\equiv E_{\ell+1}\mod \ell$, we have
\begin{align}\label{r4s6t}
\omega(E_{\ell+1}^{r}E_{4}^{s}E_{6}^{t})\equiv r+4s+6t \equiv 0\mod\ell.
\end{align}
Solving the system of congruences given by (\ref{r4s6t}) and (\ref{coeff q}) yields
\begin{align}
7r &\equiv -72t \mod\ell\label{r is},\\
14s &\equiv 15t \mod\ell\label{s is}.
\end{align}
Substituting (\ref{r is}) and (\ref{s is}) into 49 times (\ref{coeff qq}) yields
\begin{align*}
-8255520t\equiv 0 \mod \ell.
\end{align*}
Since $8255520=2^{5}\cdot 3^{4}\cdot 5\cdot 7^{2}\cdot13 $, the lemma follows.
\end{proof}

\section{Proof of Theorem~\ref{main thm}}
We begin with the trivial observation that $E_{2}^{r}E_{4}^{s}E_{6}^{t} = 1 + \cdots$ does not have a simple congruence at $0\mod\ell$.  Hence, we assume that $E_{2}^{r}E_{4}^{s}E_{6}^{t}$ has a simple congruence at $c\not\equiv 0\mod\ell,$ where $\ell\geq 5$.  Since $E_{2}\equiv E_{\ell+1}\mod \ell$, $E_{\ell+1}^{r}E_{4}^{s}E_{6}^{t}$ has a simple congruence at $c\mod\ell$.  Recall that our goal is to show $\ell \leq 2r +8|s|+12|t| + 21$.  Hence, if $\ell < |s|$ or $\ell< |t|$ then we are done.  Thus we assume $\ell+s\geq 0$ and $\ell+t\geq 0$.  We also assume $\ell > 11$.  Lemma~\ref{theta is zero} allows us to take $\Theta(E_{2}^{r}E_{4}^{s} E_{6}^{t})\not\equiv 0\mod\ell$ (otherwise we are done).  By Lemma~\ref{replacement lemma} we see that
\begin{align*}
E_{\ell+1}^{r}E_{4}^{\ell+s}E_{6}^{\ell+t} \in M_{(r+10)\ell + (r+4s+6t)}
\end{align*}
has a simple congruence at $c\mod\ell$.  By Lemma~\ref{eis filt},
\begin{align}\label{eqn: filt A,B}
\omega(E_{\ell+1}^{r}E_{4}^{\ell+s}E_{6}^{\ell+t}) = (r+10)\ell + (r+4s+6t).
\end{align}

We break into four cases depending on the size of $r+4s+6t$:

\begin{enumerate}
\item If $\ell \leq |r+4s+6t|$ then we are done.  

\item If $0< r+4s+6t < \ell$ then by Equation~(\ref{eqn: filt A,B}) and the first inequality of Lemma~\ref{key bounds lemma}, $\frac{\ell+1}{2} \leq r+4s+6t$ and we are done.

\item If $r+4s+6t = 0$, then by Lemma~\ref{filtration of theta f}
\begin{align*}
\omega(\Theta E_{\ell+1}^{r}E_{4}^{\ell+s}E_{6}^{\ell+t}) = (r+11)\ell +1 -s'(\ell-1)
\end{align*}
for some $1\leq s'$.  If $\ell > r+ 13$, then in order for this filtration to be non-negative, $s' \leq r+11$.  Now $\omega(\Theta E_{\ell+1}^{r}E_{4}^{\ell+s}E_{6}^{\ell+t}) \equiv s'+1\mod \ell$.  By Lemma~\ref{unex cong lem}, there must be a high point of the Tate cycle before $\Theta^{\frac{\ell+1}{2}} E_{\ell+1}^{r}E_{4}^{\ell+s}E_{6}^{\ell+t}$.  Hence
\begin{align*}
s'+1 + \frac{\ell-3}{2} \geq \ell.
\end{align*}
That is, $\ell \leq 2s' -1 \leq 2r + 21$ and we are done.

\item If $-\ell < r+4s + 6t < 0$, then take $B=\ell + r +4s+6t$ and $A=r+9$.  Equation~(\ref{eqn: filt A,B}) and the second inequality of Lemma~\ref{key bounds lemma} gives
\begin{align*}
\ell + r +4s+6t \leq r+9 + \frac{\ell+3}{2}
\end{align*}
which is equivalent to $\ell \leq 21-8s-12t$ and we are done.
\end{enumerate}

\begin{rem}\label{bound improvement}
Combining these four cases and recalling the assumptions above, we see that if $r+4s+6t > 0$ then
\begin{align*}
\ell \leq \max\{ |s|-1, |t|-1, 11, 2r+8s+6t - 1\}
\end{align*}
and if $r+4s+6t\leq 0$ then
\begin{align*}
\ell \leq \max\{ |s|-1, |t|-1, 11, 21-8s-12t\}
\end{align*}
\end{rem}

\bibliographystyle{plain}
\bibliography{/Users/Michael/Documents/bibliog}
\end{document}